\documentclass[11pt]{amsart}
\usepackage{amsfonts}
\usepackage{amssymb}
\usepackage[T1]{fontenc}
\usepackage[toc]{appendix}
\usepackage{amsmath}
\usepackage{hyperref}
\numberwithin{equation}{section}
\usepackage{dsfont}
\usepackage{color}
\usepackage{comment}
\usepackage{graphicx}
\newtheorem{theorem}{Theorem}[section]
\newtheorem{lemma}[theorem]{Lemma}
\newtheorem{proposition}[theorem]{Proposition}
\newtheorem{corollary}[theorem]{Corollary}

\theoremstyle{definition}
\newtheorem{df}{Definition}
\newtheorem{example}[df]{Example}
\newtheorem{remark}[df]{Remark}
\newtheorem{problem}[df]{Problem}

\newcommand{\N}{\mathbb N}
\newcommand{\Q}{\mathbb Q}
\newcommand{\R}{\mathbb R}

\newcommand{\ve}{\varepsilon}

\newcommand{\ov}{\overline}

\newcommand{\B}{\mathcal{B}}
\newcommand{\Pe}{\mathcal{P}}
\newcommand{\eL}{\mathcal{L}}
\newcommand{\eps}{\varepsilon}
\newcommand{\stronglyto}{\overset{d}{\boldsymbol{\to}}}
\makeatletter
\@namedef{subjclassname@2020}{%
  \textup{2020} Mathematics Subject Classification}
\makeatother


\begin{document}
\author{Piotr Nowakowski}
\address{P. Nowakowski:\newline
Faculty of Mathematics and Computer Science, University of \L \'{o}d\'{z},
Banacha 22, 90-238 \L \'{o}d\'{z}, Poland\newline
Institute of Mathematics, Czech Academy of Sciences,
\v{Z}itn\'a 25, 115 67 Prague 1, Czech Republic}
\email{piotr.nowakowski@wmii.uni.lodz.pl}

\author{{Filip Turobo\'s}}
\address{{
F. Turobo\'s:\newline
Institute of Mathematics, \L \'{o}d\'{z} University of Technology,
al. Politechniki 10, 93-590 \L \'{o}d\'{z}, Poland}}
\email{{ filip.turobos@p.lodz.pl}}

\title[Applications of ball spaces theory]{Applications of ball spaces theory: fixed point theorems in semimetric spaces {and ball convergence}}
%
\begin{abstract}
In the paper we apply some of the results from the theory of ball spaces in semimetric setting. This {allows} us to obtain fixed point theorems which we believe to be unknown to this day. 
As a byproduct, we obtain the equivalence of some different notions of {completeness} in semimetric spaces where the distance function is $1$-continuous. In the second part of the article, we generalize {the} Caristi-Kirk results for $b$-metric spaces. Additionally, we obtain {a} characterization of semicompleteness for $1$-continuous $b$-metric spaces via a fixed point theorem analogous to a result of Suzuki. In the epilogue, we introduce the concept of convergence in ball spaces, based on the idea that balls should resemble closed sets in topological sets. We prove several of its properties, compare it with convergence in semimetric spaces and pose several open questions connected with this notion.
\end{abstract}

\maketitle
\section{Introduction}
The notion of a ball space, that is, a pair of a set and a nonempty family of some of its nonempty subsets, first appeared in the paper \cite{KK1}. It was used to prove some fixed point theorems in metric, ultrametric and topological spaces. This idea was continued in \cite{BCLS,CKK,KKP}. In our paper we will generalize some results obtained in the mentioned articles to the scope of semimetric spaces.

Fixed point theorems are tools which are used in many various areas of mathematics as well as in other fields. Since semimetric spaces, i.e. spaces with a distance function which does not have to satisfy the triangle inequality, turned out to be important in applications for example in computer science (e.g. see the section Applications in \cite{LNT}), also fixed point theorems in such spaces became important. This topic was investigated for example in \cite{ADKR,BP,CJT,Su,M}.  
However, using the theory of ball spaces in this area is a novelty (except examining metric and ultrametric spaces {with real-valued ultrametrics}, which are obviously semimetric spaces as well).

{The paper is organized as follows. In Section 2 we present the notation along with important definitions. In Section 3 we {give a characterization of completeness in semimetric spaces}, using the ball space theory. Following this train of thought, in Section 4 we present some generalizations of the Caristi-Kirk fixed point theorem for b-metric spaces. Finally, in Section 5, we introduce a notion of ball convergence. We show its connection {with} the classical notions of convergence in topological spaces and semimetric spaces.
}

\section{Preliminaries}
Denote by $\R_+$ the set $[0, \infty)$.
Let $X$ be a set and consider a function $d \colon X \times X \to \R_+$ satisfying the following conditions:
\begin{itemize}
\item[(S1)] $\forall_{x,y \in X} \, d(x,y) = 0 \Leftrightarrow x=y;$

\item[(S2)] $\forall_{x,y \in X} \, d(x,y) = d(y,x).$
\end{itemize}

Then $d$ is called a semimetric and $(X,d)$ a semimetric space. 

In metric spaces we have the additional condition called the triangle inequality:

\begin{itemize}
\item[(M)] $\forall_{x,y,z \in X} \, d(x,z) \leq d(x,y) + d(y,z).$
\end{itemize}

In semimetric spaces we will consider some weaker \textit{"triangle-like"} conditions. Let us begin by introducing some necessary notions.

Take $g \colon \R_+ \times \R_+ \to \R_+$. We say that $g$ is nonreducing if for any $a,b \in \R_+$ we have $g(a,b) \geq \max\{a,b\}$. We say that $g$ is nondecreasing if for any $a,b,c ,d \in \R_+$ such that $a \leq c, b \leq d$ we have $g(a,b) \leq g(c,d)$. We call $g$ amenable if $g(a,b)  = 0$ if and only if $a=b=0$.

A function $h\colon X \to (-\infty,+\infty]$ is called proper if $\{x\in X\colon h(x) \in \R \} \neq \emptyset$.

A function $g \colon \R_+ \times \R_+ \to \R_+$ is called semitriangular if it is nonreducing, nondecreasing, continuous at $(0,0)$ and amenable. We define the semimetric triangle condition (G) generated by a semitriangular mapping $g$ as follows:

\begin{itemize}
\item[(G)] $\forall_{x,y,z \in X} \, d(x,z) \leq g(d(x,y), d(y,z)).$
\end{itemize}

If $d$ is a semimetric satisfying condition (G) for some semitriangular function $g:\R_+\times\R_+\to\R_+$, then the space $(X,d)$ will be called a (G)-semimetric space (compare with the definition of regular semimetric in \cite{BP}{, where there is no assumption that $g$ is nonreducing, but it is assumed that $g$ is symmetric and increasing with respect to each argument}). Whenever we will be writing about a semimetric triangle condition (G), we will assume that it is generated by a semitriangular function $g \colon \R_+ \times \R_+ \to \R_+$. 

Take a (G)-semimetric space $(X,d)$, $x \in X$ and $r > 0$. Set
$$B_r(x):= \{y \in X\colon d(x,y) \leq r\}$$ 
and
$${B^{\circ}_r(x)}:= \{y \in X\colon d(x,y) < r\}.$$ 

We define limits, Cauchy sequences and completeness in semimetric spaces analogously as in metric spaces (although it is important to keep in mind that the properties of these concepts may not be the same as in the metric case). {We will also need the notions of semicompleteness of semimetric spaces from \cite[Definition 6]{Su} and $1$-continuity of a semimetric from \cite{GS}}. We say that a semimetric space is semicomplete if every Cauchy sequence has a convergent subsequence. A semimetric $d$ is said to be $1$-continuous if for any sequence $(x_n)$ in $X$ convergent to some $x \in X$ we have 
$$\forall_{y \in X} \,\,\lim\limits_{n \to \infty} d(x_n,y) = d(x,y).$$
Lastly, a semimetric $d$ is uniformly $1$-continuous if{,} for any sequence $(x_n)$ in $X$ convergent to some $x \in X${,} we have 
$$\forall_{\ve >0} \exists_{N \in \N} \forall_{n\geq \N} \forall_{y \in X} \,\, |d(x_n,y)-d(x,y)|<\ve.$$

For any set $X$ we will denote by $\Pe(X)$ its power set, i.e. the family of all subsets of $X$.

Now, we will define a ball space. Consider a nonempty set $X$ and take its nonempty subset $\B \subset {\Pe(X)\setminus \{\emptyset\}}$ . A pair $(X, \B)$ is called a ball space and the elements of $\B$ are called balls. Any chain (that is, nonempty totally ordered subset) of $(\B ,\subset)$ is called a nest of balls. We say that a ball space is spherically complete if every nest of balls has a nonempty intersection.

Let $(X,d)$ be a (G)-semimetric space. Take any $S \subset (0,\infty)$. Put
$$\B_S := \{B_r(x) \colon x \in X, r \in S\}.$$

For any set $A \subset \R$ we denote by $A^d$ the set of all accumulation points of $A$.
\section{{Completeness of semimetric spaces}}

In this section we will {give a characterization of complete $1$-continuous (G)-semimetric spaces using the notion of spherical completeness of ball spaces. It is a generalization of Theorem 5.6 from \cite{CKK}, where the characterization of complete metric spaces is given.} The proof is similar.
\begin{theorem} \label{zup}
Let $(X,d)$ be a (G)-semimetric space, where $d$ is $1$-continuous. Then the following conditions are equivalent:
\begin{itemize}
\item[(1)] $(X,d)$ is semicomplete;

\item[(2)] $(X,d)$ is complete;

\item[(3)] for any $S \subset (0, \infty)$ such that $S^d = \{0\}$ the ball space $(X,\B_S)$ is spherically complete;

\item[(4)] there exists $S \subset (0, \infty)$ such that $S^d = \{0\}$ and the ball space $(X,\B_S)$ is spherically complete.
\end{itemize}
\end{theorem}
\begin{proof}
(1) $\Rightarrow$ (2) \\
Take any Cauchy sequence $(x_n)$ of elements of $X$. Since $(X,d)$ is semicomplete, there is a subsequence $(x_{n_k})$ which is convergent to some $x \in X$. Take $\ve > 0$. By {the} continuity of $g$ at $(0,0)$ and the fact that $g(0,0)=0$, there is $\delta >0$ such that for any $v \in (0,\delta]$ we have $g(v,v) < \ve$. Pick $N_1 \in \N$ such that $d(x_{n_k},x) < \delta$ for all $k \geq N_1$. Choose $N_2 \in \N$ such that $d(x_n,x_m) < \delta$ for all $n,m \geq N_2$. Put $N {:=} \max\{N_1,N_2\}.$ Then, $n_N \geq N$. Using the fact that $g$ is nondecreasing, we have for $n \geq N$
$$d(x_n,x) \leq g(d(x_n,x_{n_N}),d(x_{n_N},x)) \leq g(\delta,\delta) < \ve.$$
Hence $\lim\limits_{n\to \infty}x_n = x$, so $(X,d)$ is complete.
\\
(2) $\Rightarrow$ (3)

Assume that $(X,d)$ is complete and let $S\subset {(0,\infty)}$ be such that $S^d=\{0\}$ (it can be shown that a subset of $\R$ with a single accumulation point is at most countable). Let $\eL := \{B_t\colon t \in T\}$ be a nest of balls, where $(T,\leq)$ is some  totally ordered set and for $t \in T$, $B_t:=B_{r_t}(x_t)$ for some $r_t\in S$ and $x_t\in X$. We can assume that if $s,t \in T$ are such that $s \leq t$, then $B_t \subset B_s$. Consider the following two cases:
\begin{itemize}
\item There exists $t_0\in T$ such that for all $t\geq t_0$ we have $r_t\geq r$ for some fixed $r>0$.
Since for all $t\leq t_0$ we have $B_t\supset B_{t_0}$, we have that $\bigcap_{t\in T} B_t = \bigcap_{t\geq t_0} B_t$. 
Assume that $\bigcap_{t\geq t_0} B_t = \emptyset$. In particular, this implies the existence of $t_1>t_0$ such that $x_{t_0}\notin B_{t_1}$. Therefore, we have the following two inequalities
\[
r_{t_1}<d\left(x_{t_0},x_{t_1}\right)\leq r_{t_0}.
\]
The first inequality stems from the fact that $x_{t_0}\notin B_{r_{t_1}}(x_{t_1})$ and the latter comes from the inclusion $ B_{r_{t_1}}(x_{t_1}) \subset B_{r_{t_0}}(x_{t_0})$, which, in particular, implies $x_{t_1} \in B_{r_{t_0}}(x_{t_0})$. We can now proceed inductively -- having defined $x_{t_n}$ as the center of the $t_n$-th ball, if $x_{t_n}\notin \bigcap_{t\geq t_{n}} B_{t_n}$, then there exists $t_{n+1}$ such that $x_{t_n}\notin B_{r_{t_{n+1}}}(x_{t_{n+1}})$. Then, {the} same argumentation as previously proves that 
\[
r_{t_{n+1}}<d\left(x_{t_n},x_{t_{n+1}}\right)\leq r_{t_n}.
\]
As a result, we end up with a strictly decreasing sequence $(r_{t_n})_{n\in\N}$ whose elements are bounded from below by $r$. As such, $(r_{t_n})$ converges to some $r'\geq r>0$. Since the sequence $(r_{t_n})_{n\in\N}$ is strictly decreasing, $r'$ is an accumulation point of $S$. Thus, we have arrived at {a} contradiction, as $S^d=\{0\}$. Thus, $\bigcap_{t\in T} B_t \neq \emptyset$.
\item There exists a sequence $(t_n)_{n\in\N}$ such that $r_{t_n}\to 0$ and for every $t\in T $ there exists $n\in\N$ such that $t_n>t$. If such situation occurs, the sequence $(x_{t_n})_{n\in\N}$ is a Cauchy sequence. Indeed, if we have indices $n<m$, then 
\[
d(x_{t_n},x_{t_m})\leq r_{t_n} \xrightarrow{n,m\to \infty} 0,
\]
where the first inequality follows from {the} inclusion $B_{r_{t_n}}(x_{t_n}) \supset B_{r_{t_m}}(x_{t_m})$. As $(X,d)$ is complete, the discussed sequence tends to some $x\in X$. From {the} $1$-continuity of $d$ we then obtain that for any fixed $n_0\in \N$ we have
\[
d(x,x_{t_{n_0}})= \lim_{k\to\infty} d(x_{t_k},x_{t_{n_0}}) \leq r_{t_{n_0}},
\]
hence $x\in B_{t_n}$ for each $n$. One can easily see that $x\in \bigcap_{t\in T} B_t$, since for every $t\in T$ there exists $n\in\N$ such that
\[
B_t \supset B_{t_n} \ni x.
\]
Thus, the intersection is not empty as it contains at least one element $x$.
\end{itemize} 
\noindent
(3) $\Rightarrow$ (4) Obvious. \\
(4) $\Rightarrow$ (1) 

Consider $S \subset (0, \infty)$ such that $S^d = \{0\}$ and assume that the ball space $(X,\B_S)$ is spherically complete. Fix a Cauchy sequence $(x_n)$ in $X$. We will inductively define a sequence $(s_n)$ in $S$ such that $s_n \geq g(s_{n+1},s_{n+1})$ for $n \in \N$. Take $s_1 \in S$. Assume that for some $n \in \N$ we have defined $s_i$ for $i \leq n$. Using {the} continuity of $g$ at $(0,0)$ and the fact that $g(0,0)=0${,} we can find $\delta >0$ such that for any $v \in (0,\delta]$ we have $g(v,v) \leq s_n$. Since $0$ is an accumulation point of $S$, we can find $s_{n+1} \in S$ such that $s_{n+1} \leq t$. Then $s_n \geq g(s_{n+1},s_{n+1})$.

Using induction, we will now define an increasing sequence of natural numbers $(n_i)$. Choose $n_1 \in \N$ in such a way that $d(x_n,x_m) \leq s_2$ for all $n,m \geq n_1$. Assume that for some $j \in \N$ we have defined $n_i$ for $i \leq j$. Pick $n_{j+1} > n_j$ such that $d(x_n,x_m) \leq s_{j+2}$ for all $n,m \geq n_{j+1}$. For $i \in \N$ denote 
$$B_i:=B_{s_i}(x_{n_i}).$$
We will show that the balls $B_i$ form a nest. Fix $i \in \N$ and $y \in B_{i+1}$. Then $d(y,x_{n_{i+1}}) \leq s_{i+1}$. We also have that $d(x_{n_i},x_{n_{i+1}}) \leq s_{i+1}$, because $n_i,n_{i+1} \geq n_i.$ Using the fact that $g$ is nondecreasing, we obtain
$$d(x_{n_i},y) \leq g(d(x_{n_i},x_{n_{i+1}}),d(y,x_{n_{i+1}})) \leq g(s_{i+1},s_{i+1}) \leq s_i.$$
Hence $y \in B_{s_i}(x_{n_i}) = B_i$, so $B_{i+1} \subset B_i$. Since $(X,\B_S)$ is spherically complete{,} $\bigcap_{i\in \N} B_i \neq \emptyset$. Take $x \in \bigcap_{i\in \N} B_i$. Then $x \in B_i$ for all $i \in \N$, and thus $d(x_{n_i},x) \leq s_i$. Hence $\lim\limits_{i\to \infty}d(x_{n_i},x) = 0$ because $s_i \to 0$. Therefore, $(x_{n_i})$ is a convergent subsequence of $(x_n)$, which finally implies that $(X,d)$ is semicomplete.   
\end{proof}

\section{Generalizations of the Caristi-Kirk Fixed Point Theorem}
In this part of the paper, we will present several generalizations of {the} celebrated Caristi-Kirk fixed point theorem. We start by introducing some additional, necessary notions and then move to new results. The proofs in this section are based on the methods used in \cite{KKP} and \cite{BCLS}.

In this section we will consider one of the most popular semimetric spaces, that is, b-metric spaces. This notion was due to Bakhtin \cite{Bakhtin} but some references point out Czerwik \cite{Czerwik} as the author of this concept. Take $K \geq 1.$ We say that a semimetric space $(X,d)$ is a b-metric space with a constant $K$ if
$$\forall_{x,y,z \in X} \,\,d(x,z)\leq K(d(x,y)+d(y,z)).$$
So, b-metric spaces are (G)-semimetric spaces with the function $g$ of the form: $g(a,b) = K(a+b).$ Of course, for $K=1$ we obtain a metric space.

Similarly as in \cite{KKP} or \cite{BCLS} we will define the Caristi-Kirk and Oettli-Th\'era ball spaces. However, we will adjust them to b-metric spaces.
Consider a mapping $\phi \colon X \to \R$. We say that $\phi$ is sequentially lower semicontinuous if{,} for every $y \in X$ and for any sequence $(y_n)$ convergent to $y${,} we have
$$\liminf\limits_{n\to \infty} \phi(y_n) \geq \phi(y).$$
A function $\phi$ is called a Caristi-Kirk function, if it is sequentially lower semicontinuous and bounded from below. 
Given any Caristi-Kirk function $\phi$ and $x \in X$, we define the $K$-Caristi-Kirk balls as the sets of the form:
$$B^{\phi}_x:= \{y \in X \colon d(x,y) \leq K\phi(x)-K\phi(y)\}.$$
Since obviously $x \in B^{\phi}_x$, the sets $B^{\phi}_x$ are nonempty.
So, we can consider the $K$-Caristi-Kirk ball space $(X, \B^{\phi}),$ where
$$\B^{\phi}:=\{B^{\phi}_x\colon x \in X\}.$$

By $\ov{\R}$ we denote the set $\R\cup \{+\infty\}.$
We say that a function $\Phi\colon X \times X \to \ov{\R}$ is a $K$-Oettli-Th\'era function if:
\begin{itemize}
\item[(i)] for every $x \in X$ the function $\Phi(x,\cdot)\colon X \to \R$ is sequentially lower {semicontinuous};

\item[(ii)] $\forall_{x \in X} \,\,\Phi(x,x) = 0$;

\item[(iii)] $\forall_{x,y,z \in X} \,\,\Phi(x,z) \leq K(\Phi(x,y)+\Phi(y,z))$;

\item[(iv)] $\exists_{x_0 \in X} \,\, \inf\limits_{x\in X} \Phi(x_0,x) > - \infty.$
\end{itemize}
Every element $x_0$ satisfying (iv) will be called an Oettli-Th\'era element for $\Phi$. 

Given any $K$-Oettli-Th\'era function $\Phi$ and $x \in X${,} we define the $K$-Oettli-Th\'era balls as the sets of the form:
$$B^{\Phi}_x:= \{y \in X \colon d(x,y) \leq -\Phi(x,y)\}.$$
The sets $B^{\Phi}_x$ are nonempty (see Lemma \ref{lem1}).
So, we can consider the $K$-Oettli-Th\'era ball space $(X, \B^{\Phi}),$ where
$$\B^{\Phi}:=\{B^{\Phi}_x\colon x \in X\}.$$

For a fixed Oettli-Th\'era element $x_0$ for $\Phi$ we also define the ball space $(B^{\Phi}_{x_0},\B^{\Phi}_{x_0}),$ where
$$\B^{\Phi}_{x_0} = \{B^{\Phi}_y\colon y \in B^{\Phi}_{x_0}\}.$$

We will need a notion of a strongly contractive ball space. We say that a ball space $(X,\B)$ is strongly contractive if there exists a family of balls $\{B_x \in \B\colon x \in X\}$ such that for any $x,y \in X$ the following conditions hold:
\begin{itemize}
\item[(1)] $x \in B_x$;

\item[(2)] $y \in B_x\setminus\{x\} \Rightarrow B_y \subsetneq B_x$.
\end{itemize}

We conclude this preliminary part with the following fixed point result for spherically complete, strongly contractive ball spaces.

\begin{theorem} \cite[{Theorem 2}]{BCLS} \label{twosc}
Let $(X,\B)$ be a spherically complete, strongly contractive ball space. Then for any $x \in X$ there is $a \in B_x$ such that $B_a = \{a\}$.
\end{theorem}

The following lemma provides some insight into $K$-Oettli-Th\'{e}ra ball spaces.

\begin{lemma} \emph{(}cf. \cite[Lemma 10]{BCLS}\emph{)} \label{lem1}
Let $(X,d)$ be a b-metric space with a constant $K \geq 1$ and let $\Phi \colon X \times X \to \ov{\R}$  be such that $\Phi(x,x) =0$ for all $x \in X$ and
$$\forall_{x,y,z \in X} \,\,\Phi(x,z) \leq K(\Phi(x,y)+\Phi(y,z)).$$ For any $x \in X$ define $B_x:= \{y \in X\colon d(x,y) \leq - \Phi(x,y)\}.$ Then for every $x \in X$ we have:
\begin{itemize}
\item[(1)] $x \in B_x$;

\item[(2)] if $y \in B_x$, $y \neq x$, then $B_y \subsetneq B_x$ and $\Phi(x,y)<\Phi(y,x)$.
\end{itemize}
\end{lemma}
\begin{proof}
Fix $x \in X$.

Ad (1) We have $d(x,x) =0 = -\Phi(x,x)$, so $x \in B_x$.

Ad (2) Fix $y \in B_x$ and $z \in B_y$. Then 
$$d(x,y) \leq - \Phi(x,y)$$
and 
$$d(y,z) \leq - \Phi(y,z).$$
By the assumption on $\Phi$, we have 
$$d(x,z) \leq K(d(x,y)+d(y,z)) \leq K(-\Phi(x,y)-\Phi(y,z)) = -K(\Phi(x,y)+\Phi(y,z)) \leq - \Phi(x,z).$$
Hence $z \in B_x$, and consequently{,} $B_y \subset B_x$.

Now, assume additionally that $y\neq x$. We will prove that $x \notin B_y$. On the contrary, suppose that $x \in B_y$. Then $d(y,x) \leq -\Phi(y,x)$. We have
$$0 < K(d(x,y)+d(y,x)) \leq K(-\Phi(x,y)-\Phi(y,x)) = - K(\Phi(x,y)+\Phi(y,x)) \leq -\Phi(x,x) = 0,$$
a contradiction. Hence $x \notin B_y$. Moreover, we proved that $$-\Phi(y,x) < d(y,x)=d(x,y) \leq -\Phi(x,y).$$ 
\end{proof}

From this result we {almost immediately obtain} the following

\begin{corollary} \label{wstrcon}
Let $(X,d)$ be a b-metric space with a constant $K \geq 1$ and $\Phi \colon X\times X \to \ov{\R}$ be a $K$-Oettli-Th\'era function. Then the space $(X,\B^{\Phi})$ is strongly contractive. Moreover, for any Oettli-Th\'era element $x_0$ for $\Phi$, the ball space $(B^{\Phi}_{x_0},\B^{\Phi}_{x_0})$ is also strongly contractive and $\B^{\Phi}_{x_0}=\{B \in B^{\Phi}\colon B \subset B^{\Phi}_{x_0}\}.$
\end{corollary}
\begin{proof}
The first assertion follows directly from Lemma \ref{lem1}, when we consider the family $\{B^{\Phi}_x\in \B^{\Phi}\colon x \in X\}$. By $(2)$ of Lemma \ref{lem1}, for any $x \in  B^{\Phi}_{x_0}$ we have $B_x^{\Phi} \subset B_{x_0}^{\Phi}$, so $\B^{\Phi}_{x_0}=\{B \in B^{\Phi}\colon B \subset B^{\Phi}_{x_0}\}.$ Using again Lemma \ref{lem1}, we see that the family $\{B^{\Phi}_x\in \B^{\Phi}_{x_0}\colon x \in B^{\Phi}_{x_0}\}$ witnesses that the ball space $(B^{\Phi}_{x_0},\B^{\Phi}_{x_0})$ is also strongly contractive.
\end{proof}

The subsequent result highlights the connection between Caristi-Kirk functions and $K$-Oettli-Th\'{e}ra mappings.

\begin{lemma} \label{CK a OT}
Let $(X,d)$ be a b-metric space with a constant $K \geq 1$ and let $\phi \colon X \to \R$ be a Caristi-Kirk function. Then the function $\Phi \colon X \times X \to \R$ given by the formula $\Phi(x,y) = K\phi(y)-K\phi(x)$ is a $K$-Oettli-Th\'era function and $B^{\phi}_x = B^{\Phi}_x$ for any $x \in X$. Moreover, every $x \in X$ is an Oettli-Th\'era element for $\Phi$.
\end{lemma}
\begin{proof}
The proofs {of} conditions (i) and (ii) from the definition of $K$-Oettli-Th\'era function are immediate. Fix $x,y,z \in X$. We have
$$\Phi(x,z) = K\phi(z)-K\phi(x) = K\phi(z)-K\phi(y)+K\phi(y) - K\phi(x) = K\Phi(y,z) + K\Phi(x,y).$$
Condition (iv) holds for any $x_0$, because the codomain of $\phi$ is $\R$ and $\phi$ is bounded from below.
Pick any $x \in X$. Then we have
$$B^{\Phi}_x = \{y \in X \colon d(x,y) \leq - \Phi(x,y)\} = \{y \in X \colon d(x,y) \leq K\phi(x) - K\phi(y)\} = B^{\phi}_x.$$
\end{proof}

\begin{corollary} \label{wn lem dla CK}
Let $(X,d)$ be a b-metric space with a constant $K \geq 1$ and let $\phi \colon X \to \R$ be a Caristi-Kirk function.
If $y \in B^{\phi}_x$, then $B^{\phi}_y \subset B^{\phi}_x$.
\end{corollary}
\begin{proof}
Define $\Phi(x,y) = K\phi(y) - K\phi(x)$ for $x,y \in X$. By Lemma \ref{CK a OT}, $\Phi$ is a $K$-Oettli-Th\'era function, and so it satisfies the assumptions of Lemma \ref{lem1}. By Lemma \ref{CK a OT}, we have
$$B^{\Phi}_x =B^{\phi}_x.$$
Hence, by (2) of Lemma \ref{lem1}, we obtain the assertion.
\end{proof}
The following proposition is the natural counterpart of Theorem \ref{zup} for $K$-Caristi-Kirk ball spaces.

\begin{proposition} \emph{(}cf. \cite[Proposition 3]{KKP} \emph{)} \label{CKzup}
Let $(X,d)$ be a b-metric space with a constant $K \geq 1$, where $d$ is $1$-continuous. Then: 
\begin{itemize}
\item[(1)] If $(X,d)$ is semicomplete, then every $K$-Caristi-Kirk ball space is spherically complete.

\item[(2)] If $d$ is uniformly $1$-continuous and every $K$-Caristi-Kirk ball space is spherically complete, then $(X,d)$ is complete.
\end{itemize}
\end{proposition}
\begin{proof}
Ad (1) Assume that $(X,d)$ is semicomplete. By Theorem \ref{zup}, $(X,d)$ is complete. Let $\phi\colon X \to \R$ be a Caristi-Kirk function. Consider any nest of balls $\eL$ in $\B^{\phi}$. By the definition of $\B^{\phi}$, there exists $M \subset X$ such that $\eL = \{B^{\phi}_x\colon x \in M\}$. For any $x,y \in M$ we have $B^{\phi}_x \subset B^{\phi}_y$ or $B^{\phi}_x \subset B^{\phi}_y$. Since $x \in B^{\phi}_x$ for all $x \in X$, we have that $x \in B^{\phi}_y$ or $y \in B^{\phi}_x.$ In both cases, 
\begin{equation} \label{nier}
d(x,y) \leq K|\phi(x)-\phi(y)|.
\end{equation}
Put $r:= \inf\limits_{x\in M} \phi(x)$.
Since $\phi$ is bounded from below, $r \in \R$.
Take a sequence $(x_n)$ in $M$ such that $\lim\limits_{n \to \infty} \phi(x_n) =r.$ The sequence $(\phi(x_n))$ is a Cauchy sequence in $\R$ because it is convergent. Take $\ve > 0$ and choose $N \in \N$ such that $$|\phi(x_n)-\phi(x_m)| \leq \frac{\ve}{K}$$ for $n,m \geq N$. By (\ref{nier}), we have for $n,m \geq N$
$$d(x_n,x_m) \leq K|\phi(x_n)-\phi(x_m)| \leq \ve.$$
So, $(x_n)$ is a Cauchy sequence in $X$. By {the} {completeness} of $(X,d)$, $(x_n)$ is convergent to some $z \in X$. We will show that $z \in \bigcap \eL$. 
From {the} sequential lower semicontinuity of $\phi$ we infer that
$$\phi(z) \leq \lim\limits_{n \to \infty} \phi(x_n) = r.$$
Take $x \in M$. We will show that $z \in B^{\phi}_x$. By (\ref{nier}) and {the} $1$-continuity of $d$, we have
$$d(x,z) = \lim\limits_{n\to \infty} d(x,x_n) \leq  \lim\limits_{n\to \infty} K|\phi(x)-\phi(x_n)|=K|\phi(x)-r| = K\phi(x)-Kr \leq K\phi(x) - K\phi(z),$$
so $z \in B^{\phi}_x$. By {the} {arbitrariness} of $x$, $z \in \bigcap \eL$. Thus, $(X,\B^{\phi})$ is spherically complete.

Ad (2) Let $(x_n)$ be a Cauchy sequence in $(X,d)$. If one of its terms is its limit, then the proof is finished, so assume otherwise. Define $\psi \colon X \to \R$ according to the following formula
$$\psi(x) = \limsup\limits_{n \to \infty} d(x,x_n).$$
We will show that $\psi$ is well defined, that is, it cannot be equal to $\infty$ for any $x$. Take $x \in X$ and pick $N \in \N$ such that $d(x_n,x_m) \leq 1$ for all $n,m \geq N$. Then for any $n \geq N$ we have
$$d(x,x_n) \leq Kd(x,x_N)+Kd(x_N,x_n) \leq Kd(x,x_N) + K.$$
Hence the sequence $(d(x,x_n))$ is bounded from above, and {therefore} $$\limsup\limits_{n \to \infty} d(x,x_n)<\infty.$$

Now, using induction, we will choose a subsequence $(x_{n_k})$ of the sequence $(x_n)$ as follows. Put $n_1:=1$. Assume that we have defined $n_k$ for some $k \in \N$. Since $x_{n_k}$ is not a limit of $(x_n)$, $\psi(x_{n_k}) > 0.$ However, $\lim\limits_{n\to \infty} \psi(x_n) =0$, because $(x_n)$ is a Cauchy sequence. Hence
$$\limsup\limits_{n\to \infty} \left(\frac{1}{2}d(x_{n_k},x_n) + \psi(x_n)\right)=\limsup\limits_{n\to \infty} \frac{1}{2}d(x_{n_k},x_n) + \lim\limits_{n \to \infty} \psi(x_n) = \frac{1}{2}\psi(x_{n_k}) \leq \psi(x_{n_k}).$$ 
Therefore, there exists $m >n_k$ such that
\begin{equation} \label{1/2d}
\frac{1}{2}d(x_{n_k},x_m) \leq \psi(x_{n_k})-\psi(x_m).
\end{equation}
Put $n_{k+1}:=m$. Define $\phi\colon X \to \R$ by the formula:
$$\phi(x)= 2\psi(x).$$
By the definition, $\phi$ is bounded from below by $0$. We will prove that $\psi$ is sequentially lower semicontinuous. Fix a positive $\varepsilon$. Take $y \in X$ and let $(y_n)$ be a sequence convergent to $y$. Take an increasing sequence of natural numbers $(m_k)$ such that $\lim\limits_{k \to \infty}d(y,x_{m_k}) = \psi(y)$ and $$|d(y,x_{m_k})-\psi(y)|< \frac{\ve}{2}$$ for all $k \in \N$. Fix $N \in \N$ such that 
$$|d(y_n,x)-d(y,x)|< \frac{\ve}{2}$$
for all $x \in X$ and $n \geq N$. 
Then we have
$$|d(y_n,x_{m_k}) - \psi(y)| \leq |d(y_n,x_{m_k})- d(y,x_{m_k})| + |d(y,x_{m_k}) - \psi(y)| < \ve$$ 
for all $n \geq N$ and $k \in \N$. Hence
$$\lim\limits_{n \to \infty} \psi(y_n) = \lim\limits_{n \to \infty} \limsup\limits_{m \to \infty} d(y_n,x_m) \geq \lim\limits_{n \to \infty} \lim\limits_{k \to \infty} d(y_n,x_{m_k}) = \psi(y).$$
Therefore, $\psi$ is sequentially lower semicontinuous, and so is $\phi$. Thus, $\phi$ is a Caristi-Kirk function. By the assumption, the $K$-Caristi-Kirk ball space $(X, \B^{\phi})$ is spherically complete. Put $$\eL:=\{B^{\phi}_{x_{n_k}}\colon k \in \N\}.$$
By (\ref{1/2d}), we have
$$d(x_{n_k},x_{n_{k+1}})\leq 2\psi(x_{n_k})-2\psi(x_{n_{k+1}}) = \phi(x_{n_k})-\phi(x_{n_{k+1}}) \leq K(\phi(x_{n_k}) - \phi(x_{n_{k+1}}))$$
for all $k\in \N$, because $\phi(x_{n_k})-\phi(x_{n_{k+1}}) \geq 0$. Hence $x_{n_{k+1}} \in B^{\phi}_{x_{n_{k}}}$, and by Corollary \ref{wn lem dla CK}, $B^{\phi}_{x_{n_{k+1}}} \subset B^{\phi}_{x_{n_k}}.$ In consequence, $\eL$ is a nest of balls. From {the} spherical completeness of $(X, \B^{\phi})$ we deduce that there exists $x \in \bigcap \eL.$ Thus,
$$d(x_{n_k},x) \leq \phi(x_{n_k})-\phi(x)\leq \phi(x_{n_k})$$
for all $k\in \N$. Since $\lim\limits_{k\to \infty}\phi(x_{n_k}) =0$, the sequence $(x_{n_k})$ is convergent to $x$. Hence $(X,d)$ is semicomplete. By Theorem \ref{zup} $(X,d)$ is complete.
\end{proof}

We proceed with a bit more technical lemma, which will be put to use in the subsequent part of the paper.

\begin{lemma} \emph{(}cf. \cite[ Lemma 13]{BCLS}\emph{)} \label{lem2}
Let $(X,d)$ be a b-metric space with a constant $K\geq 1$, $\Phi\colon X \times X \to \ov{\R}$ {be} a $K$-Oettli-Th\'era function and $x_0$ {be} an Oettli-Th\'era element for $\Phi$. Let $\eL \subset \B^{\Phi}$ be a nest of balls of the form
$$\eL = \{B_x \colon x \in A\},$$ where $B_x = B^{\Phi}_x,$ and $A \subset B^{\Phi}_{x_0}.$ Then for every $x, y \in A$ we have
\begin{equation}\label{nlem}
d(x,y) \leq |\Phi(x_0,x)-\Phi(x_0,y)|.
\end{equation}
Moreover, the following conditions are equivalent for any $x,y \in A$:
\begin{itemize}
\item[(i)]
$y \in B_x$;
\item[(ii)]
$\Phi(x,y)\leq \Phi(y,x)$;
\item[(iii)]
$\Phi(x_0,y) \leq \Phi(x_0,x)$.
\end{itemize}
\end{lemma}
\begin{proof}
First, observe that for any $x \in A$, $\Phi(x_0,x) \leq 0$. Indeed, since $A \subset B^{\Phi}_{x_0}$, we have $$0 \leq d(x_0,x) \leq -\Phi(x_0,x).$$ Pick $x,y \in A$. 
Since $\eL$ is a nest, either $x \in B_y$ or $y \in B_x$. Hence $d(x,y) \leq -\Phi(x,y)$ or $d(x,y) \leq -\Phi(y,x).$
Without loss of generality assume that $d(x,y) \leq -\Phi(x,y)$.
By condition (iii) from the definition of a $K$-Oettli-Th\'era function, we have 
\begin{equation} \label{1/k}
\frac{1}{K} \Phi(x_0,y) \leq \Phi(x_0,x) + \Phi(x,y).
\end{equation}

(i) $\Leftrightarrow$ (ii) Assume that $y \in B_x$. If $y=x$, then (ii) obviously holds. If $y \neq x$, then by Lemma~\ref{lem1}~(3), $-\Phi(y,x) < -\Phi(x,y)$, and so we have (ii).
If (i) is not satisfied, that is, $y \notin B_x$, then $x \in B_y$ and $x \neq y$. Using once more Lemma \ref{lem1} (3), we obtain $-\Phi(y,x) > -\Phi(x,y)$, so (ii) does not hold.

(i) $\Leftrightarrow$ (iii) Assume that $y \in B_x$.
Then, by (\ref{1/k}) and the fact that $\Phi(x_0,y)\leq 0$, we have
$$0 \leq d(x,y)\leq -\Phi(x,y) \leq \Phi(x_0,x)- \frac{1}{K}\Phi(x_0,y) \leq \Phi(x_0,x)- \Phi(x_0,y).$$
So, (iii) holds. 
If (i) is not satisfied, that is, $y \notin B_x$, then $x \in B_y$ and $x \neq y$. Using once more (\ref{1/k}) (swapping $x$ and $y$) and the fact that $\Phi(x_0,x)\leq 0$, we obtain 
$$0 < d(x,y)\leq -\Phi(y,x) \leq \Phi(x_0,y)- \frac{1}{K}\Phi(x_0,x) \leq \Phi(x_0,y)- \Phi(x_0,x),$$ so (iii) does not hold. We have also proved that (\ref{nlem}) holds for any $x,y \in A$.
\end{proof}

In the sequel we will need yet another version of Theorem \ref{zup} and Proposition \ref{CKzup}. This time we approach the problem of completeness from the perspective of $K$-Oettli-Th\'{e}ra functions.

\begin{proposition} \label{propzup}\emph(cf. \cite[Proposition 14]{BCLS}\emph)
Let $(X,d)$ be a b-metric space with a constant $K\geq 1$, where $d$ is $1$-continuous. Then
\begin{itemize}
\item[(1)] if $(X,d)$ is semicomplete, then the space $(B^{\Phi}_{x_0},\B^{\Phi}_{x_0})$ is spherically complete for every $K$-Oettli-Th\'era function $\Phi$ and every Oettli-Th\'era element $x_0$ for $\Phi$.
\item[(2)] if $d$ is uniformly $1$-continuous and the space $(B^{\Phi}_{x_0},\B^{\Phi}_{x_0})$ is spherically complete for every $K$-Oettli-Th\'era function $\Phi$ and every Oettli-Th\'era element $x_0$ for $\Phi$, then $(X,d)$ is complete;
\end{itemize}
\end{proposition}
\begin{proof}
Ad (1)
Assume that $(X,d)$ is semicomplete. By Theorem \ref{zup}, $(X,d)$ is complete. Let $\Phi\colon X \times X\to \ov{\R}$ be a $K$-Oettli-Th\'era function and $x_0 \in X$ be an Oettli-Th\'era element for $\Phi$. Consider a nest of balls $\eL$ in $\B^{\Phi}_{x_0}$. By the definition of $\B^{\Phi}_{x_0}${,} there exists $M \subset B^{\Phi}_{x_0}$ such that $\eL = \{B^{\Phi}_x\colon x \in M\}$. 
Put $r:= \inf\limits_{x\in M} \Phi(x_0,x)$.
Since $x_0$ is an Oettli-Th\'era element for $\Phi$, we know that $r \in \R$.
Take a sequence $(x_n)$ in $M$ such that $\lim\limits_{n \to \infty} \Phi(x_0,x_n) =r.$ The sequence $(\Phi(x_0,x_n))$ is a Cauchy sequence in $\R$ because it is convergent. By (\ref{nlem}) from Lemma \ref{lem2}, we have
$$d(x_n,x_m) \leq |\Phi(x_0,x_n)-\Phi(x_0,x_m)|$$
for any $n,m \in \N$. Hence $(x_n)$ is a Cauchy sequence in $X$. 
By {the} {completeness} of $(X,d)$, $(x_n)$ is convergent to some $z \in X$. We will show that $z \in \bigcap \eL$. 
Take any $x \in M$. We will show that $z \in B^{\phi}_x$. 
We consider the two following cases:

$1^{\mbox{o}}$ $\Phi(x_0,x) = r$. Using (\ref{nlem}) from Lemma \ref{lem2}, we obtain
$$d(x_n,x)\leq |\Phi(x_0,x_n)-\Phi(x_0,x)| = |\Phi(x_0,x_n)-r|\to 0.$$
By {the} $1$-continuity of $d$, we have $d(x,z)=0$, so $x=z$. Therefore{,} $z \in B^{\Phi}_x$. 

$2^{\mbox{o}}$ $\Phi(x_0,x) > r$. Then there is $N \in \N$ such that 
$$\Phi(x_0,x_n) \leq \Phi(x_0,x)$$
for all $n \geq N.$ From Lemma \ref{lem2} we deduce that $$\Phi(x,x_n) \leq \Phi(x_n,x)$$
for any $n \in \N$. 
By the definition of $\eL$ and the fact that $x,x_n \in M$, we have $x \in B^{\Phi}_{x_n}$ or $x_n \in B^{\Phi}_{x}$, so $d(x,x_n) \leq -\Phi(x,x_n)$ or $d(x,x_n) \leq -\Phi(x_n,x)$. Hence
$$d(x,x_n) \leq \max\{-\Phi(x,x_n),-\Phi(x_n,x)\} = -\Phi(x,x_n).$$
Using {the} above inequality, the $1$-continuity of $d$ and the sequential lower semicontinuity of $\Phi(x,\cdot)$ we obtain
$$d(x,z) = \lim\limits_{n\to \infty} d(x,x_n) = \limsup\limits_{n\to \infty} d(x,x_n) \leq \limsup\limits_{n\to \infty}\left(-\Phi(x,x_n)\right) = -\liminf\limits_{n\to \infty}\Phi(x,x_n) \leq - \Phi(x,z).$$
{Hence} $z \in B^{\Phi}_x$. By {the} arbitrariness of $x$, $z \in \bigcap \eL$. Thus, $(B^{\Phi}_{x_0},\B^{\Phi}_{x_0})$ is spherically complete.

Ad (2) Let $\phi$ be a Caristi-Kirk function. Consider the ball space $(X,\B^{\phi})$ and take any nest of balls $\eL$ in that space. Take $x_0 \in X$ such that $B^{\phi}_{x_0} \in \eL$. Let $\Phi \colon X \times X \to \R$ be given by the formula $\Phi(x,y) := K\phi(y)-K\phi(x).$ By Lemma \ref{CK a OT}, $\Phi$ is a $K$-Oettli-Th\'era function and $x_0$ is an Oettli-Th\'era element for $\Phi$. Consider the nest
$$\eL_0 = \{B^{\phi}_{y} \in \eL \colon B^{\phi}_{y}  \subset B^{\phi}_{x_0}\}.$$ Of course, $\bigcap \eL = \bigcap \eL_0$. By Lemma \ref{CK a OT}, we have 
$$\eL_0 = \{B^{\Phi}_{y} \in \eL \colon B^{\Phi}_{y}  \subset B^{\Phi}_{x_0}\}.$$
Hence $\eL_0$ is a nest in the ball space $(B^{\Phi}_{x_0},\B^{\Phi}_{x_0})$, which is spherically complete, by the assumption. Thus,
$$\emptyset \neq \bigcap \eL_0 = \bigcap \eL,$$
and so $(X,\B^{\phi})$ is spherically complete. From the arbitrariness of $\phi$ and Proposition \ref{CKzup} we infer that $(X,d)$ is complete.
\end{proof}

We can now present the initial fixed point result in the introduced setting.

\begin{proposition} \label{Propklu} \emph(cf. \cite[Proposition 16]{BCLS}\emph)
Let $(X,d)$ be a semicomplete b-metric space with a constant $K \geq 1$, where $d$ is $1$-continuous. 
\begin{itemize}
\item[(1)] If $\Phi \colon X \times X \to \ov{\R}$ is a $K$-Oettli-Th\'era function, then for every Oettli-Th\'era element $x_0$ for $\Phi$ there exists $a \in B^{\Phi}_{x_0}$ such that $B_a^{\Phi} = \{a\}.$

\item[(2)] If $\phi \colon X \to R$ is a Caristi-Kirk function, then for every $x \in X$ there exists $a \in B^{\phi}_{x}$ such that $B_a^{\phi} = \{a\}.$
\end{itemize}
\end{proposition}
\begin{proof}
Ad (1) Let $\Phi \colon X \times X \to \ov{\R}$ be a $K$-Oettli-Th\'era function and $x_0$ be an Oettli-Th\'era element for $\Phi$. By Proposition \ref{propzup}, $(B^{\Phi}_{x_0},\B^{\Phi}_{x_0})$ is spherically complete and by Corollary \ref{wstrcon}, it is strongly contractive. From Theorem \ref{twosc} we have the assertion.

Ad (2) Follows directly from (1) and Lemma \ref{CK a OT}.
\end{proof}

From above proposition we can infer the Caristi-Kirk type fixed point theorems. 

\begin{theorem}\emph(cf. \cite[Theorem 21]{BCLS}\emph) \label{CK th}
Let $(X,d)$ be a semicomplete b-metric space with a constant $K \geq 1$, where $d$ is $1$-continuous and $f \colon X \to X$. Let $\Phi \colon X \times X \to \ov{\R}$ be a $K$-Oettli-Th\'era function. If
\begin{equation} \label{CKwar}
\forall_{x\in X} \,\, d(x,f(x)) \leq - \Phi(x,f(x)),
\end{equation}
then there is $a \in X$ such that $f(a)=a.$
\end{theorem}
\begin{proof}
From (\ref{CKwar}) we infer that $f(x) \in B_x$ for any $x \in X$. By Proposition \ref{Propklu}, there is $a \in X$ such that $B_a = \{a\}$. Since $f(a) \in B_a$, we have $f(a)=a$.
\end{proof}
The result above should be compared with \cite[Corollary 12]{Su}, which presents a similar contribution, albeit in slightly different direction. Before we present this result we need some additional definitions from \cite[Definition 6]{Su}. We say that $(x_n) \in X^\N$ is a $\Sigma$-Cauchy sequence if $\sum_{n=1}^\infty d(x_n,x_{n+1}) < \infty.$ A semimetric space $(X,d)$ is called $\Sigma$-semicomplete if every $\Sigma$-Cauchy sequence has a convergent subsequence. 
\begin{lemma} \cite[{Proposition 8}]{Su} \label{sigma}
Let $(X,d)$ be a $\Sigma$-semicomplete semimetric space. Then $X$ is semicomplete.
\end{lemma}
\begin{theorem}\cite[Theorem 13]{Su} \label{chara}
Let $(X,d)$ be a semimetric space. Then the following are equivalent:
\begin{itemize}
\item[(i)] $X$ is $\Sigma$-semicomplete;

\item[(ii)] every function $f \colon X\to X$ has a fixed point whenever there is a proper, sequentially lower semicontinuous function $h \colon X \to [0,\infty]$ such that 
$$\forall_{x\in X} \,\, d(x,f(x)) \leq h(x)-h(f(x)).$$
\end{itemize}

\end{theorem}
Using Theorem \ref{CK th}, we can prove the following analogue to the previous result.
\begin{theorem}
Let $(X,d)$ be a b-metric space with a constant $K \geq 1$, where $d$ is $1$-continuous. Then the following are equivalent:
\begin{itemize}
\item[(i)] $X$ is semicomplete;

\item[(ii)] every function $f \colon X\to X$ has a fixed point whenever there is a proper, sequentially lower semicontinuous function $h \colon X \to [0,\infty]$ such that 
$$\forall_{x\in X} \,\, d(x,f(x)) \leq h(x)-h(f(x));$$

\item[(iii)] $X$ is $\Sigma$-semicomplete.
\end{itemize}
\end{theorem}
\begin{proof}
(i) $\Rightarrow$ (ii). Take $f \colon X\to X$. Assume that there exists a proper, sequentially lower semicontinuous function $h \colon X \to [0,\infty]$ such that 
$$\forall_{x\in X} \,\, d(x,f(x)) \leq h(x)-h(f(x)).$$ Define a mapping $\Phi \colon X\times X \to \ov{\R}$ by the formula $\Phi(x,y) := h(y)-h(x)$. We will show that $\Phi$ is a $K$-Oettli-Th\'era function. The proofs of conditions (i) and (ii) from the definition of $K$-Oettli-Th\'era function are immediate. Pick any $x,y,z \in X$. We have
$$\Phi(x,z) = h(z)-h(x) = h(z)-h(y)+h(y) - h(x) = \Phi(y,z) + \Phi(x,y)\leq K(\Phi(y,z) + \Phi(x,y)).$$ 
Choose $x_0 \in \{x \in X\colon h(x) \in \R\}$. Such $x_0$ exists because $h$ is proper. Then for any $x \in X$, we have $\Phi(x_0,x) = h(x)-h(x_0) \geq -h(x_0).$ Hence $\Phi$ is a $K$-Oettli-Th\'era function. Since we have 
$$\forall_{x\in X} \,\, d(x,f(x)) \leq h(x)-h(f(x))= -\Phi(x,f(x)),$$
by Theorem \ref{CK th}, $f$ has a fixed point.

(ii) $\Rightarrow$ (iii)
Follows from Theorem \ref{chara}. 

(iii) $\Rightarrow$ (i)
Follows from Lemma \ref{sigma}.
\end{proof}

{
Now, we return from this slight detour to give a quick series of generalizations of the {Caristi}-Kirk theorem in various settings. {The} reasoning in the proofs are similar to the ones presented in \cite{BCLS}. However, {the} obtained results are more general.}

\begin{theorem}\emph(cf. \cite[Theorem 22]{BCLS}\emph)
Let $(X,d)$ be a semicomplete b-metric space with a constant $K \geq 1$, where $d$ is $1$-continuous and $F \colon X \to \Pe(X)$. Let $\Phi \colon X \times X \to \ov{\R}$ be a $K$-Oettli-Th\'era function. If
\begin{equation} \label{CKwar2}
\forall_{x\in X} \,\exists_{y \in F(x)}\,\, d(x,y) \leq - \Phi(x,y),
\end{equation}
then there is $a \in X$ such that $a \in F(a).$
\end{theorem}
\begin{proof}
From (\ref{CKwar2}) we infer that for any $x\in X$ there is $y \in F(x)$ such that $y \in B_x$. By Proposition~\ref{Propklu}, there is $a \in X$ such that $B_a = \{a\}$. Hence there is $y \in F(a)$ such that $y \in B_a$, so $a \in F(a)$.
\end{proof}
\begin{theorem}\emph(cf. \cite[Theorem 23]{BCLS}\emph)
Let $(X,d)$ be a semicomplete b-metric space with a constant $K \geq 1$, where $d$ is $1$-continuous. Let $\Phi \colon X \times X \to \ov{\R}$ be a $K$-Oettli-Th\'era function. There exists $a \in X$ such that
\begin{equation} \label{CKwar3}
\forall_{x\in X\setminus\{a\}} \,\, d(a,x) > - \Phi(a,x).
\end{equation}
\end{theorem}
\begin{proof}
Condition (\ref{CKwar3}) means that for any $x\in X\setminus\{a\}$, $x\notin B_a$, that is $B_a = \{a\}.$ The existence of such $a$ follows from Proposition \ref{Propklu}.
\end{proof}
\begin{theorem}\emph(cf. \cite[Theorem 24]{BCLS}\emph)
Let $(X,d)$ be a semicomplete b-metric space with a constant $K \geq 1$, where $d$ is $1$-continuous. Let $\Phi \colon X \times X \to \ov{\R}$ be a $K$-Oettli-Th\'era function. For any $\gamma > 0$ and any Oettli-Th\'era element $x_0$ for $\Phi$ there exists $a \in X$ such that
\begin{equation} \label{CKwar4}
\forall_{x\in X\setminus\{a\}} \,\, \gamma d(a,x) > - \Phi(a,x)
\end{equation}
and
\begin{equation} \label{CKwar4'}
\gamma d(x_0,a) \leq - \Phi(x_0,a).
\end{equation}
\end{theorem}
\begin{proof}
It is easy to see that a function $\Psi \colon X \times X \to \ov{\R}$ given by the formula $\Psi(x,y) = \frac{1}{\gamma}\Phi(x,y)$ is also a $K$-Oettli-Th\'era function with the same Oettli-Th\'era elements as $\Phi$. Hence, by Proposition \ref{Propklu}, there is $a \in B_{x_0}^{\Psi}$ such that $B_a^{\Psi} = \{a\}$. We have
$$d(x_0,a) \leq -\Psi(x_0,a) = - \frac{1}{\gamma}\Phi(x_0,a),$$
because $a \in B_{x_0}^{\Psi}$. This gives us (\ref{CKwar4'}). Moreover, since $B_a^{\Psi} = \{a\}$,
$$\forall_{x\in X\setminus\{a\}} \,\, d(a,x) > - \Psi(a,x) = - \frac{1}{\gamma}\Phi(a,x),$$
which is equivalent to (\ref{CKwar4}).
\end{proof}

\begin{theorem}\emph(cf. \cite[Theorem 25]{BCLS}\emph)
Let $(X,d)$ be a semicomplete b-metric space with a constant $K \geq 1$, where $d$ is $1$-continuous. Let $\Phi \colon X \times X \to \ov{\R}$ be a $K$-Oettli-Th\'era function. Fix $\ve \geq 0$ and $x_0 \in X$ such that $-\ve \leq \inf\limits_{x\in X} \Phi(x_0,x)$. Then for every $\gamma > 0$ and $\delta \geq \frac{\ve}{\gamma}$ there exists $a \in X$ such that $d(a,x_0) \leq \delta$ and $a$ is the strict minimum of the function $\phi_{\gamma} \colon X \to \R$ given by the formula
$$\phi_{\gamma}(x) = \Phi(a,x) + \gamma d(x,a).$$
\end{theorem}
\begin{proof}
Pick $\gamma > 0$ and $\delta \geq \frac{\ve}{\gamma}$.
A function $\Psi \colon X \times X \to \ov{\R}$ given by the formula $\Psi(x,y) = \frac{1}{\gamma}\Phi(x,y)$ is a $K$-Oettli-Th\'era function with the same Oettli-Th\'era elements as $\Phi$. Hence, by Proposition \ref{Propklu}, there is $a \in B_{x_0}^{\Psi}$ such that $B_a^{\Psi} = \{a\}$. We have
$$d(x_0,a) \leq -\Psi(x_0,a) = - \frac{1}{\gamma}\Phi(x_0,a),$$
because $a \in B_{x_0}^{\Psi}$. 
Thus,
$$\gamma d(x_0,a) \leq -\Phi(x_0,a) \leq - \inf\limits_{x\in X} \Phi(x_0,x) \leq \ve \leq \gamma \delta.$$
So, $d(a,x_0) \leq \delta$.
Moreover, since $B_a^{\Psi} = \{a\}$,
$$\forall_{x\in X\setminus\{a\}} \,\, d(a,x) > - \Psi(a,x) = - \frac{1}{\gamma}\Phi(a,x).$$
Therefore, for any $x\in X\setminus\{a\}$, we have
$$\phi_{\gamma}(x) = \Phi(a,x) + \gamma d(a,x) > 0 = \phi_{\gamma}(a),$$
which finishes the proof.
\end{proof}
For the next result we need another definition. Consider a semimetric space $(X,d)$, $\gamma \in (0, \infty)$ and $a,b \in X$. The set 
$$P_{\gamma}(a,b):=\{y \in X \colon \gamma d(y,a) + d(y,b) \leq d(a,b)\}$$
is called a petal associated with $\gamma$ and $a,b$.
\begin{theorem}\emph(cf. \cite[Theorem 27]{BCLS}\emph)
Let $M$ be a semicomplete subset of a b-metric space $(X,d)$ with a constant $K \geq 1$, where $d$ is $1$-continuous. Fix $x_0 \in M$ and $b \in X\setminus M$. Then for every $\gamma > 0$ there exists $a \in P_{\gamma}(x_0,b) \cap M$ such that 
$$P_{\gamma}(a,b) \cap M = \{a\}.$$
\end{theorem}
\begin{proof}
Fix $\gamma > 0$.
A function $\phi \colon M \to \R$ given by the formula $\phi(x) = \frac{1}{K\gamma}d(x,b)$ is a Caristi-Kirk function on $M$. For any $x \in M$ we have 
$$P_{\gamma}(x,b) \cap M = \{y \in M \colon \gamma d(y,x) + d(y,b) \leq d(x,b)\} = \{y \in M \colon d(y,x) \leq \frac{1}{\gamma}d(x,b) -  \frac{1}{\gamma} d(y,b)\}$$$$= \{y \in M \colon d(y,x) \leq K\phi(x) -  K\phi(y)\}= B_x^{\phi},$$
where $B_x^{\phi}$ are $K$-Caristi-Kirk balls in $M$. 
Using Proposition \ref{Propklu} (2), we obtain that there exists $a \in B_{x_0}^{\phi}$ such that $B_a^{\phi} = \{a\}$. Thus,
$$\{a\} = B_a^{\phi} = P_{\gamma}(a,b) \cap M.$$
\end{proof}
\begin{theorem}\emph(cf. \cite[Theorem 28]{BCLS}\emph)
Let $(X,d)$ be a semicomplete b-metric space with a constant $K \geq 1$, where $d$ is $1$-continuous. Let $\Phi \colon X \times X \to \ov{\R}$ be a $K$-Oettli-Th\'era function and $x_0 \in X$ be an Oettli-Th\'era element for $\Phi$. If for any $y \in B_{x_0}^{\Phi}$ {satisfying} $\inf\limits_{x\in X} \Phi(y,x) < 0$ there exists $z \in X$, $z \neq y$ such that $d(y,z) \leq - \Phi(y,z),$ then there is $a \in B_{x_0}^{\Phi}$ {for which} $\inf\limits_{x\in X} \Phi(a,x) = 0$.
\end{theorem}
\begin{proof}
By Proposition \ref{Propklu}, there is $a \in B_{x_0}^{\Phi}$ such that $B_a^{\Phi} = \{a\}$. We will show that $\inf\limits_{x\in X} \Phi(a,x) = 0$. Since $\Phi(a,a) = 0$, we have $\inf\limits_{x\in X} \Phi(a,x) \leq 0$. Assume on the contrary that $\inf\limits_{x\in X} \Phi(a,x) < 0$. Then, by the assumption, there exists $z \in X$, $z \neq a$ such that $d(a,z) \leq - \Phi(a,z).$ But it means that $z \in B_a^{\Phi} = \{a\}$, a contradiction. Hence $\inf\limits_{x\in X} \Phi(a,x) = 0$.
\end{proof}
\begin{theorem}\emph(cf. \cite[Theorem 29]{BCLS}\emph)
Let $(X,d)$ be a semicomplete b-metric space with a constant $K \geq 1$, where $d$ is $1$-continuous. Let $\Phi \colon X \times X \to \ov{\R}$ be a $K$-Oettli-Th\'era function and $x_0 \in X$ be an Oettli-Th\'era element for $\Phi$. Let $A \subset X$ be such that 
$$\forall_{x\in B_{x_0}^{\Phi}\setminus A} \, \, \exists_{y \in X\setminus\{x\}}\,\, d(x,y) \leq -\Phi(x,y).$$
Then $B_{x_0}^{\Phi} \cap A \neq \emptyset$.
\end{theorem}
\begin{proof}
By Proposition \ref{Propklu}, there is $a \in B_{x_0}^{\Phi}$ such that $B_a^{\Phi} = \{a\}$. We will show that $a\in A$. On the contrary assume that $a \notin A$. Then, by the assumption, there exists $y \in X$, $y \neq a$ such that $d(a,y) \leq -\Phi(a,y)$. But it means that $y \in B_a^{\Phi} = \{a\}$, a contradiction. Hence $a \in B_{x_0}^{\Phi} \cap A$.
\end{proof}

\section{Some remarks on the notion of ball-convergence and topology of ball spaces}

{It seems natural to ask whether the notion of convergence can be introduced in the setting of ball spaces}. The instinctiveness of such question may come from the fact, that we have already some convergence-related notions. A good example would be the spherical completeness, which strongly resembles completeness from the metric setting.

Let us consider a ball space $(X,\B)$ and a sequence of its elements $(x_n)_{n\in\N}$. We will say that $(x_n)_{n\in\N}$ $\B$-converges
to some $x\in X$ if $ \bigcap \B_{(x_n)} = \{x\}$, where $\B_{(x_n)}$ is the family of all balls from $\B$ containing infinitely many terms from sequence $(x_n)$, i.e., 
\begin{equation}\label{ballconverg}
\B_{(x_n)}:= \left\{ B\in \B \, : \, \operatorname{card}\left( \{x_n \, : \, n\in\N\} \cap B\right) = \aleph_0 \right\}.
\end{equation}

This fact will be denoted by $x_n\overset{\B}{\to} x$.

{\begin{theorem} \label{top}
Let $(X,\tau)$ be a $T_0$ topological space. Consider the ball space $(X,\B)$, where $\B$ is the family of all closed sets in $(X,\tau)$. Consider a sequence $(x_n)$ of elements of $X$ and a point $x \in X$. Then $(x_n)$ is convergent to $x$ with respect to $\tau$ if and only if $x_n\overset{\B}{\to} x$.
\end{theorem}
\begin{proof}
"$\Leftarrow$". Assume that $x_n\overset{\B}{\to} x$. This means that $x$ belongs to every closed set containing infinitely many terms of $(x_n)$. Take $U$ as an arbitrary open neighbourhood of $x$. Suppose that $U$ does not contain almost all terms of $(x_n)$. Hence the closed set $X\setminus U$ contains infinitely many terms of $(x_n)$, so due to $\B$-convergence we have $x \in X\setminus U$, a contradiction. Therefore, $(x_n)$ is convergent to $x$ with respect to $\tau$.

"$\Rightarrow$". Assume that $(x_n)$ converges to $x$ with respect to $\tau$, that is, any open neighbourhood of $x$ contains all but finitely many terms of $(x_n)$. Let $B$ be a closed set containing infinitely many terms of $(x_n)$. Suppose that $x \notin B$. Then $x \in X\setminus B$ and this set is open. By the assumption, $X\setminus B$ contains all but finitely many terms of $(x_n)$, and so $B$ contains finitely many terms of $(x_n)$, a contradiction. Now, suppose that there is $y \in X$ such that for any closed set $B$ containing infinitely many terms of $(x_n)$, $y \in B$. Repeating the reasoning from "$\Leftarrow$", we obtain that $(x_n)$ converges to $y$ with respect to $\tau$. Since $(X,\tau)$ is $T_0$, the limits are unique, so $x=y$, which finishes the proof.
\end{proof}}

\begin{remark}
In the reasoning above, instead of picking $\B$ as the family of all closed sets, one can take a topological basis $\mathcal{U}$ of $(X,\tau)$ and define $\B$ as the complements of sets from $\mathcal{U}$.
\end{remark}

As we can take any family of {nonempty} sets in the role of $\B$, the notion of $\B$-convergence covers a large variety of other \textit{modes} of convergence which guarantee the uniqueness of the limit. 

Having defined a notion of convergence in ball spaces, one can ask the question how this relates to the standard convergence in semimetric spaces (where $x_n\to x$ if and only if $d(x_n,x)\to 0$ as $n\to\infty$). Take any semimetric space $(X,d)$. Consider the topology $\tau_\B$ generated by the convergence introduced by the ball space $(X,\B_S)$ for some $S\subset (0,+\infty)$. For these two notions of convergence to coincide, we first need to guarantee that $0$ is an accumulation point of $S$. If that is not the case, the sequence $(x_n)$ might be $\B_S$-convergent to some $x$ even if all its terms are $\varepsilon$-apart from $x$, as balls might be unable to detect such small gaps between points. 

The $1$-continuity of a semimetric guarantees that standard convergence implies $\B$-convergence. This claim is formally proved in the following theorem.

\begin{theorem}\label{g1convergenceb}
Let $(X,d)$ be any $1$-continuous (G)-semimetric space. If {a} sequence $(x_n)_{n\in\N}$ of elements of $X$ converges to $x\in X$, it also $\B_S$-converges to $x$ in the ball space $(X,\B_S)$, where $S=(0,+\infty)$.
\end{theorem}
\begin{proof}
Assume that $(x_n)_{n\in\N}$ converges to $x$. Then for each $s>0$ there exists $n_0\in \N$ such that $x_n\in B_s(x)$ for $n\geq n_0$.
Clearly{, $\bigcap_{s\in S} B_s(x)=\{x\}$}, so it is enough to show that for any $y\in X$ and any $s\in S$ such that $B_s(y)$ contains infinitely many terms of $(x_n)$ we have $x\in B_s(y)$. Let $(x_{n_k})$ be a subsequence of $(x_n)$ contained in $B_s(y)$. Clearly $d(x_{n_k},x)\to 0$ and from {the} $1$-continuity of $d$ we have
\[
d(x,y) = \lim_{k\to\infty} d(x_{n_k},y)\leq s.
\]
Hence $x\in B_s(y)$ which finishes the proof.

\end{proof}
Straightforwardly from the proof of Theorem \ref{g1convergenceb} we obtain the following: 

\begin{corollary}
Let $S\subset (0,+\infty)$ be such that $0 \in S^d$ and take any $1$-continuous (G)-semimetric space $(X,d)$. If the sequence $(x_n)_{n\in\N}$ of elements of $X$ converges to $x\in X$, it also $\B_S$-converges to $x$~in the ball space $(X,\B_S)$.
\end{corollary}

Unfortunately, we cannot reverse this implication as we have to face the following example in a~metric setting:
{\begin{example}
Consider a collection of distinct points $x,x_1,x_2,\dots$. Put $X:=\{x,x_1,x_2,\dots\}$. Define a metric on $X$ as follows
$$d(y,z)=d(z,y) := \left\{ \begin{array}{ccc}
0\;\text{ if }\; y=z \\ 
1 \;\text{ if }\; (y=x \vee z=x) \wedge y\neq z\\
1+\frac{1}{|n-m|} \;\text{ if }\; y=x_n,z=x_m, n\neq m.%
\end{array}%
\right. $$
The function $d$ is indeed a metric, because {the} first two axioms follow from the definition of $d$ and {the} triangle inequality follows easily from the fact that $d(X \times X) \subset \{0\} \cup [1,2]$. Take $\B$ as a family of all closed balls in $(X,d)$. Obviously, $d(x_n,x) \not\to 0$. However, we will show that $x_n\overset{\B}{\to} x$. Let {$B_r(y)$} be a closed ball containing infinitely many terms of $(x_n)$. We will show that $x \in {B_r(y)}$. Consider the cases:

\begin{itemize}
\item[$1.$] $y=x$. Then obviously $x \in {B_r(y)}$. 

\item[$2.$] $y=x_n$ for some $n \in \N$. Since $d(x_n,x_m) > 1$ for any $m \neq n$, we have $r \geq 1$. Because $d(x,x_n) =1,$ we obtain $x \in {B_r(y)}$.
\end{itemize}

Now, observe that for any $m \in \N$ there exists a ball ${B_r(y)}$ containing infinitely many terms of $(x_n)$ such that $x_m \notin {B_r(y)}$. Indeed, choose $y = x_{m+1}$, $ r = \frac{3}{2}$. Then, $d(x_m,x_{m+1}) = 1+1 = 2 > \frac{3}{2},$ so $x_m \notin {B_\frac{3}{2}(x_{m+1})}.$ On the other hand, for any $n > m+2$ we have $d(x_{m+1},x_n) =1 +\frac{1}{|n-m-1|} \leq \frac{3}{2}$. Thus, $x_n \in {B_\frac{3}{2}(x_{m+1})}.$ Finally, $x_n\overset{\B}{\to} x$.
\end{example}}

As the notion of $\B$-convergence even in metric spaces turns out to be weaker than standard convergence, we pose the following question: under what conditions the implication from Theorem \ref{g1convergenceb} can be reversed?

A partial answer to this question is provided by the doubling property. A semimetric space is said to be $N$-doubling if every closed ball of radius $r>0$ can be covered by at most $N$ closed balls of radius $\frac{r}{2}$ (one can think of such condition as a metric-type finite-dimensionality).

\begin{theorem}\label{convergo}
Let $(X,d)$ be a $1$-continuous semimetric space with $N$-doubling property for some $N\in\N$ and $S={(0,\infty)}$. If {a} sequence $(x_n)_{n\in\N}$ $\B_S$-converges to some $x\in X$ in {the} ball space $(X,\B_S)$, then $(x_n)$ contains a subsequence $(x_{n_k})$ such that $d({x_{n_k}},x)\to 0$ as $n\to\infty$. {Moreover, if $(x_n)$ is bounded, then $d(x_{n},x)\to 0$ as $n\to\infty$.}
\end{theorem}
\begin{proof}
Our assumption can be restated as follows: 
\begin{center}
\textit{For every $y\in X$ and $r\in S$, if $B_r(y)$ contains infinitely many terms of $(x_n)${,} then $d(x,y)\leq r$.}
\end{center}
We {want to find a subsequence $(x_{n_k})$ of a sequence $(x_n)$ such that $d(x_{n_k},x)\to 0$}. Define
\[
r_0:=\inf \{ r>0 \, : \, B_r(x)  \mbox{ contains infinitely many terms of } (x_n)\}.
\]
At first, let us note that $r_0$ is finite. Indeed, there exists at least one ball of the form $B_s(y)$ (for some $y\in X$ and $s>0$) which contains infinitely many terms of $(x_n)$ as well as $x$. Hence for all $n\in \N$ such that $x_n\in B_s(y)$ we have:
\[
d(x_n,x)\leq g\left(d(x,y),d(y,x_n) \right) \leq g\left(d(x,y),s\right) =:\hat{r}.
\]

Therefore, $r_0\leq \hat{r}$ as $B_{\hat{r}}(x)$ contains infinitely many terms of $(x_n)$. Suppose that $r_0>0$. Thus, there exists a subsequence $(x_{n_k})$ of $(x_n)$ which satisfies $d(x_{n_k},x)\in \left[r_0,\frac{3}{2}r_0\right]$ (otherwise, $r_0$ could be lowered further). The doubling property guarantees that $B_{\frac{3}{2}r_0}(x)$ can be covered by a family of $N$ balls $\{B_{\frac{3}{4}r_0}(y_i^{(1)}) \, : \, i\leq N\}$. In particular, at least one such ball (denote it by $B_{\frac{3}{4}r_0}(y^{(1)})$) contains infinitely many elements of $(x_{n_k})$. From our $\B_S$-convergence assumption, such a ball contains $x$ as well.

This ball can also be covered by another set of balls with half the radius, i.e., balls of the form $B_{\frac{3}{8}r_0}(y_i^{(2)})$. One such ball is guaranteed to contain infinitely many terms of $(x_n)$, denote it by $B_{\frac{3}{8}r_0}(y^{(2)})$. It contains $x$ as well.

Proceeding inductively, we obtain a sequence of balls $B_{\frac{3}{2^{k+1}}r_0}(y^{(k)})$, each of these containing infinitely many terms of $(x_n)$ as well as $x$. Thus, $d(x,y^{(k)}) \to 0$. For all $n\in \N$ such that $x_n\in B_{\frac{3}{2^{k+1}}r_0}(y^{(k)})$ we have:
\begin{equation}\label{klopf}
d(x,x_n)\leq g\left(d(x,y^{(k)}),d(x_n,y^{(k)})\right)\leq g\left(\frac{3}{2^{k+1}}r_0,\frac{3}{2^{k+1}}r_0\right)=:s_k.
\end{equation}
As $\frac{3}{2^{k+1}}r_0\to 0$, we have $s_k\to 0$ from the continuity of $g$ {at} the origin. But since $s_k\to 0$, we have that $s_k<r_0$ for sufficiently large $k$. Thus $B_{s_k}(x)$ contains infinitely many terms of $(x_n)_{n\in\N}$, which contradicts our assumption that $r_0>0$. {Finally, pick $n_1=1$ and for $k >1$ define $n_{k} > n_{k-1}$ such that $x_{n_k} \in B_{\frac{1}{k}}(x)$. Since $r_0 = 0${,} we can always find such terms. We obtain $d(x_{n_k},x) \to 0$.}

{Now, assume that $(x_n)$ is bounded. We have that there is $y \in X$ and $s>0$ such that $B_s(y)$ contains all elements of $(x_n)$.

Suppose on the contrary that $d(x_n,x) \not\to 0$, that is, there are $r>0$ and a subsequence $(x_{n_k})$ of $(x_n)$ such that  $d(x_{n_k},x) > r$ for every $k\in \N$. 

The doubling property guarantees that $B_{s}(y)$ can be covered with a family of $N$ balls $\{B_{\frac{1}{2}s}(y_i^{(1)}) \, : \, i\leq N\}$. In particular, at least one such ball (denote it by $B_{\frac{1}{2}s}(y^{(1)})$) contains infinitely many elements of $(x_{n_k})$. From our $\B_S$-convergence assumption, such {a} ball contains $x$ as well.

Now, we can proceed similarly as in the first part of the proof to find a sequence of balls $(B_{\frac{1}{2^{j}}s}(y^{(j)}))$ each containing $x$ and infinitely many terms of $(x_{n_k})$. 
For all $k\in \N$ such that $x_{n_k}\in B_{\frac{1}{2^{j}}s}(y^{(j)})$ we have:
\begin{equation}\label{klopf2}
d(x,x_{n_k})\leq g\left(d(x,y^{(j)}),d(x_{n_k},y^{(j)})\right)\leq g\left(\frac{1}{2^{j}}s,\frac{1}{2^{j}}s\right)=:r_j.
\end{equation}
As $\frac{1}{2^{j}}s\to 0$, we have $r_j\to 0$ from the continuity of $g$ at the origin. But since $r_j\to 0$, there is $m \in \N$ such that $r_m<r$. There is also $x_{n_k} \in B_{\frac{1}{2^{m}}s}(y^{(m)})$ and, by (\ref{klopf2}), $d(x,x_{n_k}) < r$, a contradiction. Therefore, $d(x_{n},x) \to 0$.}
\end{proof}

The {theorem} above also holds even if we do not assume that $S=(0,+\infty)$ but merely require $S$ to satisfy $S^d\ni 0$. The proof gets trickier this time, because sometimes we are required to iterate the doubling property a sufficient number of times and multiply the radii of the so obtained balls by some positive constant $\rho\in(1,2)$. As $0\in S^d$, after iterating the doubling property sufficiently many times, there will be an $s_0 \in [\frac{3r_0}{2^{k+1}}, \frac{3r_0}{2^{k}})$ and we can increase the radii of our balls to $s_0$. Since balls of smaller radii are contained in the ones with same center but larger radii, we can inductively construct a sequence of balls as in the original proof. However, in our opinion this somewhat tedious reasoning can cloud the general idea behind the proof.

The following example shows that the assumption of {the} boundedness of $(x_n)$ is essential.
\begin{example}\label{PiotrusPan}
In the Euclidean metric space $(\R,d_e)$ define a sequence $(x_n)$ in the following way: for $k \in \N$ put $x_{2k} := k$ and $x_{2k-1} = \frac{1}{k}$. This sequence is obviously not convergent in $(\R,d_e)$. Consider the ball space $\B$ consisting of all compact intervals. Then $x_n\overset{\B}{\to} 0$. Indeed, it is easy to see that every compact interval containing infinitely many terms of $(x_n)$ must contain $0$. Moreover, for any $x \neq 0$ there exists an interval containing infinitely many terms of $(x_n)$ and not containing $x$ (it suffices to take $[0,\frac{|x|}{2}]$).  
(If we take as a ball space the family of all closed sets, then, by Theorem \ref{top}, we do not have ball convergence).
\end{example}

We clearly see that {the} building {of} a ball space upon some semimetric (or even metric) space $(X,d)$ using solely the closed balls leads to some pathological examples. We have asked ourselves the question, whether there exists a natural way of constructing a ball space out of $(X,d)$ which does not utilize all closed sets, but has more natural properties than $(X,\B_{(0,+\infty)})$. 

Let us consider a ball space $(X,\B_S^+)$, where $S\subseteq \R_+$ and 
\[
\B_S^+:=\B_S \cup \left\{ X \setminus { B^{\circ}_r(x)} \, : \, x \in X, r\in S \right\}.
\]

As previously, ${B^{\circ}_r(x)}:= \{ y \in X \, : \, d(x,y)<r \}$ denotes an open ball.

That is, the proposed \textit{{complemented} ball space} built upon $(X,d)$ consists of closed balls with radii in $S$, as well as the complements of open balls with radii from the same set $S$. This allows us to fix some problems similar to the one presented in Example \ref{PiotrusPan}. One can also see, that the statement of Theorem 5.3/Corollary 5.4 still holds if we replace $(X,\B_S)$ by $(X,\B_S^+)$.

\begin{theorem}
Let $(X,d)$ be a (G)-semimetric space, $(x_n)$ be a sequence in $X$ and $x\in X$. Consider a set $S\subset(0,+\infty)$ such that $0\in S^d$:
\begin{itemize}
\item[(1)]If $d$ is $1$-continuous and $(x_n)_{n\in\N}$ converges to $x$, it also {$\B_S^+$-converges} to $x$ in the {complemented} ball space $(X,\B_S^+)$.

\item[(2)] If $(x_n)_{n\in\N}$ $\B_S^+$-converges to $x$ in the {complemented} ball space $(X,\B_S^+)$, then it also converges to $x$ in $(X,d)$.
\end{itemize}
\end{theorem}
\begin{proof}
Ad (1) Let us assume that $(x_n)_{n\in\N}$ converges to $x$. As in the proof of {the} previous version of this theorem, {$\bigcap_{s\in S} B_s(x)=\{x\}$} and for any $y\in X$ and any $s\in S$ such that $B_s(y)$ contains infinitely many terms of $(x_n)$ we have $x\in B_s(y)$. Now it is enough to show { that if the complement of an open ball contains infinitely many terms of $(x_n)$, then it contains $x$ as well}. Take any $y\in X$ and $r\in S$. If {there is an infinite subsequence $(x_{k_n})_{n\in\N}$ of $(x_n)$ such that $x_{k_n} \notin {B^{\circ}_r(y)}$ for $n \in \N$}, then $d(x_{k_n},y)\geq r$ and, as a result of {the} $1$-continuity,
\[
d(x,y)=\lim_{n\to\infty} d(x_{k_n},y) \geq r.
\]
Hence, if the complement of an open ball contains infinitely many terms of $(x_{n})_{n\in\N}$, then it contains $x$ too.

Ad(2) {From the assumption that $(x_n)$ is convergent to $x$ we infer that
for every $y\in X$ and $r\in S$, if $B_r(y)$ contains infinitely many terms of $(x_n)$ then $d(x,y)\leq r$. Moreover, if for infinitely many terms of $(x_n)$ we have $d(x_n,y)\geq r$, then $d(x,y)\geq r$.}

{Now, on} the contrary, suppose that $(x_n)$ does not converge to $x$ in the semimetric sense. Then, there exists $\varepsilon>0$ such that $d(x,x_{n})\geq \varepsilon$ for infinitely many $n$. {Hence the complement of ${B^{\circ}_{\ve}(x)}$ contains infinitely many elements of $(x_{n})$. Therefore, $X\setminus {B^{\circ}_{\ve}(x)}$ should contain $x$ as well, a contradiction}.
\end{proof}

Clearly, {complemented} ball spaces seem to be a more natural ground for topological considerations. We can see, that the sequence from Example \ref{PiotrusPan} is not $\B_{(0,\infty)^+}$-convergent if we replace the ball space $(\R,\B_{(0,+\infty)})$ with the {complemented} version. We also see that the convergence in $(\R,d_d)$ (where $d_d$ stands for discrete metric) coincides with $\B_S^+$-convergence in $(\R,\B_{S}^+)$, where $S:=(0,1]$, as in both of these spaces, only the ultimately constant sequences are convergent. {However, we must be careful, because using complemented ball spaces may lead to {the} loss of spherical completeness. Therefore, a proper choice of a ball structure may vary depending on the situation.}

{In \cite{BKK}{,} the notion of ball continuity of functions between ball spaces was introduced. Let $(X,\B_x)$, $(Y,\B_Y)$ be ball spaces and consider a mapping $f\colon X \to Y$. We say that $f$ is ball continuous if for any $B \in \B_Y$ we have $f^{-1}(B) \in \B_X$. At the end of this section we would like to pose the following open problem.
\begin{problem}
Let $(X,\B_x)$, $(Y,\B_Y)$ be ball spaces and $f\colon X \to Y$. Under what assumptions on the ball spaces $(X,\B_x)$, $(Y,\B_Y)$ it is true that $f\colon X \to Y$ is ball continuous if and only if for any sequence $(x_n) \in X^{\N}$ such that   $x_n\overset{\B_X}{\to} x$, {with} $x \in X$, we have $f(x_n)\overset{\B_Y}{\to} f(x)$? 
\end{problem}
}

\appendix

\section*{Appendix}
In this appendix we want to make a slight detour from the main topic of the last section towards the notion of convergence -- but this time purely in semimetric spaces. It has been shown in several distinct papers that, in general, open balls in semimetric spaces are not necessarily open sets in the topology $\tau$ induced by sequence convergence in $(X,d)$. In particular, we refer the {reader} to \cite{ATD,PS}, as well as encourage them to see \cite[Theorem 5.6]{CJT1} which provides a huge range of such pathological spaces where no open ball is open. This is somewhat remedied by introducing the strong topology $\tau^d$ in the semimetric space $(X,d)$, whose basic open sets are the finite intersections of open balls, i.e.,
\[
\mathfrak{B}:=\left\{ { \bigcap_{i\leq n} {B^{\circ}_{r_i}(x_i)} \, : \, n\in\N, \, x_i \in X, \, r_i>0 \mbox{ for } i \leq n}\right\}. 
\]
Whenever $x_n$ converges to $x$ with respect to $\tau^d$ we will write $x_n\stronglyto x$. 
In this topology, however, $d(x_n,x)\overset{n\to \infty}{\longrightarrow} 0$ does not always imply that $x_n\stronglyto x$. We want to present this fact with an interesting example of a $b$-metric space\footnote{{This space happens to satisfy a stronger axiom, known as relaxed polygonal inequality, see \cite[Definition 12.3]{FPTIDS}.}} whose both topologies $\tau^d$ and $\tau$ are metrizable, but in distinct ways. 

\begin{example}
We shall consider the real line with the semimetric $d$ given by
\begin{equation*}
d(x,y):=\begin{cases}
|x-y|,& x,y\in \Q \, {\rm or} \, x,y\in \R\setminus \Q,\\
2|x-y|,& {\rm otherwise}.
\end{cases}
\end{equation*}
As we have the obvious Lipschitz equivalence between $d$ and the Euclidean metric $d_e$, the topology induced by sequence convergence in $(\R,d)$ coincides with the standard topology of the reals. What remains to be shown is that $\R$ equipped with $\tau^d$ is also metrizable.

Firstly, observe that an open ball in $(\R,d)$ is of the form
\[
{B^{\circ}_r(x)}:= \begin{cases}
\left(\left(x-r,x+r\right) \cap \Q \right) \cup \left(x-\frac{r}{2},x+\frac{r}{2} \right),& {\rm if }\, x\in \Q\\
\left(\left(x-r,x+r\right) \cap (\R\setminus \Q) \right) \cup \left(x-\frac{r}{2},x+\frac{r}{2} \right),& {\rm if }\, x\notin \Q.
\end{cases}
\]
Thus, each open ball resembles the planet Saturn -- we have a solid center, surrounded by a perforated ring.\footnote{The authors are aware of the fact that Saturn is a gas giant -- this fact makes the use of word \textit{solid} somewhat questionable.} Note that every open interval is also an open set in $(\R, \tau^d)$, as it can be represented as a {union} of tiny open balls which form a base of $\tau^d$. Moreover, each intersection of the form $(x-r,x+r)\cap \Q$ and $(x-r,x+r)\cap (\R\setminus\Q)$ is also an open set. To see that, pick $x\in \Q$ and consider the following intersection of open balls: \[
{B^{\circ}_{3r}(x+2r)}\cap {B^{\circ}_{3r}(x-2r)}= \left(x-r,x+r\right) \cap \Q .
\]
One can follow this reasoning to prove that all open intervals in $\Q$ (or in $\R\setminus \Q$) are open in $\tau^d$. We will now prove that $(\R,\tau^d)$ is homeomorphic to the subspace $L$ of $\R^2$ (equipped with standard metric) which is defined as
\[
L:=\left(\Q\times \{1\}\right) \cup \left( (\R\setminus \Q) \times \{0\} \right).
\]
The homeomorphism $f:\R\to L$ is given by $f(x):=(x,\chi_\Q(x))$, where $\chi_\Q$ is the indicator function of rationals. The fact that $f$ is a bijection is trivial. We will now show that both $f$ and its inverse are continuous. Fix $x\in \Q$ (the reasoning in the case where $x\in \R\setminus\Q$ is analogous). Take any open neighbourhood $A$ of $(x,1)\in L$. As $L$ inherits the natural topology from $\R^2$, there exists a rational $\varepsilon>0$ such that 
\[
A \supset ((x-\varepsilon, x+\varepsilon)\cap \Q) \times \{1\}.
\]
Put $B:=\left(x-\frac{\eps}{2},x+\frac{\eps}{2}\right)\cap \Q$. It is an open set (with respect to $\tau^d$) which satisfies $f(B)\subset A$. This proves that $f$ is continuous at $x$ -- since it was an arbitrary point, $f$ is continuous.

Let us move on to the inverse of $f$, which is $g(x,y) = x$. Fix $x \in \R$. Consider an open set $U\subset \R$ (in $\tau^d$) containing $x$. Since $U$ can be expressed as a union of base sets, we obtain:
\[
g^{-1}[U]=g^{-1}\left[\bigcup_{t\in J} \bigcap_{i\leq n_t} {B^{\circ}_{r_{t,i}}(x_{t,i})}\right]= \bigcup_{t\in J} \bigcap_{i\leq n_t} g^{-1}\left[ {B^{\circ}_{r_{t,i}}(x_{t,i})}\right],
\]
where $J$ stands for a (possibly uncountable) set of indices.
We will see that one of the intersections on the right hand side of this equality contains an image of a certain open neighbourhood of $(x,1)$. Fix $t_0 \in J$ such that $x\in \bigcap_{i\leq n_{t_0}} {B^{\circ}_{r_{t_0,i}}(x_{t_0,i})}$. Notice that if $x_{t_0,i}$ is rational, then \begin{align*}
&g^{-1}\left[{B^{\circ}_{r_{t_0,i}}(x_{t_0,i})} \right] =& \\&
\bigg(\bigg(\Q \cap \bigg( x_{t_0,i} - r_{t_0,i}, x_{t_0,i} + r_{t_0,i} \bigg) \bigg) \times \{1\}\bigg) \cup \left( \left(\left( x_{t_0,i} - \frac{1}{2}r_{t_0,i}, x_{t_0,i} + \frac{1}{2}r_{t_0,i} \right)\setminus \Q \right) \times \{0\}\right),&
\end{align*}

otherwise
\begin{align*}
&g^{-1}\left[{B^{\circ}_{r_{t_0,i}}(x_{t_0,i})} \right] =&\\ &\left(\left(\Q \cap \left( x_{t_0,i} - \frac{1}{2}r_{t_0,i}, x_{t_0,i} + \frac{1}{2}r_{t_0,i} \right) \right) \times \{1\}\right) \cup \bigg(\bigg(\bigg( x_{t_0,i} - r_{t_0,i}, x_{t_0,i} + r_{t_0,i} \bigg)\setminus \Q \bigg) \times \{0\}\bigg).&
\end{align*}

Regardless of whether $x\in \Q$ or not, $g^{-1}\left[{B^{\circ}_{r_{t_0,i}}(x_{t_0,i})} \right]$ remains {the} union of two open subsets of $L$, thus guaranteeing its openness. A finite intersection of such sets remains open, thus the preimage of $U$ is open. This proves the continuity of $f^{-1}$.

Therefore, $f$ is a homeomorphism between $L$ equipped with {the} natural topology and $(\R,\tau^d)$. This proves that $(\R,\tau^d)$ is metrizable. Since $L$ obviously is not homeomorphic to {the} real line with standard topology, we have that $\tau$ and $\tau^d$ are both metrizable in non-homeomorphic ways.
\end{example}

This example leaves us with the open question whether there exist some simple conditions under which topology $\tau^d$ is metrizable. The results which answer a similar question in the case of convergence-based topology can be found in \cite{CJT1, WAW} as well as in some of the references therein.

\section*{Acknowledgments}
Piotr Nowakowski was supported by the GA \v{C}R grant 20-22230L (Czech Science Foundation). The authors would also like to express their gratitude towards Professor Wiesław Kubi\'s, {Professor Marek Balcerzak} and Professor Jacek Jachymski for the inspiration and support. {The authors would also like to thank the referee for many relevant remarks and suggestions.}

\section*{Data Availability Statement}

The paper has no associated data.

\end{document}